\documentclass[11pt,a4paper]{amsart}
\usepackage{graphicx,amscd,color,amsmath,amsfonts,amssymb,geometry,hyperref,soul}

\newtheorem{theorem}{Theorem}
\theoremstyle{plain}

\newtheorem{lemma}{Lemma}

\newtheorem{proposition}{Proposition}
\newtheorem{remark}{Remark}

\numberwithin{equation}{section}
\begin{document}
\title[Optimal blow up rate for the constants of Khinchin type inequalities]{%
Optimal blow up rate for the constants of Khinchin type inequalities}
\author[D. Pellegrino, D. Santos, J. Santos]{Daniel Pellegrino, Djair
Santos, Joedson Santos}
\address[D. Pellegrino]{Departamento de Matem\'{a}tica \\
\indent Universidade Federal da Para\'{\i}ba \\
\indent 58.051-900 - Jo\~{a}o Pessoa, Brazil.}
\email{pellegrino@pq.cnpq.br}
\address[D. Santos]{Departamento de Matem\'{a}tica \\
\indent Universidade Federal da Para\'{\i}ba \\
\indent 58.051-900 - Jo\~{a}o Pessoa, Brazil.}
\email{djairpsc@hotmail.com}
\address[J. Santos]{Departamento de Matem\'{a}tica \\
\indent Universidade Federal da Para\'{\i}ba \\
\indent 58.051-900 - Jo\~{a}o Pessoa, Brazil.}
\email{joedson@mat.ufpb.br}
\thanks{The authors are supported by CNPq}

\begin{abstract}
We provide, among other results, the optimal blow up rate of the constants
of a family of Khinchin inequalities for multiple sums.
\end{abstract}

\maketitle

\section{Introduction}

The Khinchin inequality was designed in 1923 by A. Khinchin (\cite{kh}) to
estimate the asymptotic behavior of certain random walks. The following
example provides an illustration of its reach. Suppose that you have $n$
real numbers $a_{1},...,a_{n}$ and a fair coin. When you flip the coin, if
it comes up heads, you chose $\alpha _{1}=a_{1}$, and if it comes up tails,
you choose $\alpha _{1}=-a_{1}.$ After having flipped the coin $k$ times you
have the number
\begin{equation*}
\alpha _{k+1}:=\alpha _{k}+a_{k+1},
\end{equation*}%
if it comes up heads and
\begin{equation*}
\alpha _{k+1}:=\alpha _{k}-a_{k+1},
\end{equation*}%
if it comes up tails. After completed all $n$ steps, what should be the
expected value of
\begin{equation*}
|\alpha _{n}|=\left\vert \sum_{k=1}^{n}\pm a_{k}\right\vert ?
\end{equation*}%
Khinchin's inequality, in some sense, solves this question. Nowadays it is a very important
probabilistic tool with deep inroads in Mathematical Analysis and Banach
Space Theory. It asserts that for any $p>0$ there are constants $%
A_{p},B_{p}>0$ such that
\begin{equation}
A_{p}\left( \sum\limits_{j=1}^{n}\left\vert a_{j}\right\vert ^{2}\right) ^{%
\frac{1}{2}}\leq \left( \int\limits_{0}^{1}\left\vert
\sum\limits_{j=1}^{n}r_{j}(t)a_{j}\right\vert ^{p}dt\right) ^{\frac{1}{p}%
}\leq B_{p}\left( \sum\limits_{j=1}^{n}\left\vert a_{j}\right\vert
^{2}\right) ^{\frac{1}{2}}  \label{65}
\end{equation}%
for all sequence of scalars $\left( a_{i}\right) _{i=1}^{n}$ and all
positive integers $n.$ Above, as usual, $\left( r_{n}:[0,1]\rightarrow
\mathbb{R}\right) _{n=1}^{\infty }$ is a sequence of independent and
identically distributed random variables defined by
\begin{equation*}
r_{n}(t):=sign\left( \sin 2^{n}\pi t\right) ,
\end{equation*}%
called Rademacher functions. It is folklore that the optimal constants $%
A_{p},B_{p}$ are the same for real and complex scalars, so it suffices to
work with real scalars. It was proved by Szarek (\cite{szarek}) that $%
A_{1}=\left( \sqrt{2}\right) ^{-1}$ is optimal, solving a long standing
problem posed by Littlewood (see \cite{hall}). Later, Haagerup (\cite%
{haagerup}) simplified Szarek's approach and provided the optimal constants
for $p\neq 1$ (see also \cite{latala, tom, young}).

The Khinchin inequality is also valid -- and useful -- for multiple sums. It
is well-known (see \cite{popa}) that regardless of the choice of the
positive integers $m,n$ and scalars $a_{i_{1},\dots ,i_{m}},\,i_{1},\dots
,i_{m}=1,\dots ,n$, we have%
\begin{eqnarray}
&&\left( \sum_{i_{1},\dots ,i_{m}=1}^{n}|a_{i_{1},\dots ,i_{m}}|^{2}\right)
^{\frac{1}{2}}  \label{multik} \\
&\leq &A_{p}^{-m}\left( \int_{[0,1]^{m}}\left\vert \sum_{i_{1},\dots
,i_{m}=1}^{n}a_{i_{1},\dots ,i_{m}}r_{i_{1}}(t_{1})\cdots
r_{i_{m}}(t_{m})\right\vert ^{p}\,dt_{1}\cdots dt_{m}\right) ^{\frac{1}{p}%
}.\,  \notag
\end{eqnarray}

We stress that even in the simple case $m=2,$ the sequence of random
variables $\left( r_{i_{1}}\cdot r_{i_{2}}:[0,1]^{2}\rightarrow \mathbb{R}%
\right) _{n,m=1}^{\infty }$ is not independent.

In the present paper, among other results, we provide the exact blow up rate
of the constants in (\ref{multik}) as $n$ grows when the $\ell _{2}$-norm in
the left-hand-side is replaced by an $\ell _{r}$-norm with $0<r<2$. More
precisely, we prove the following:

\begin{theorem}
\label{7657}Let $m,n$ be positive integers and $\left( a_{i_{1},\dots
,i_{m}}\right) _{i_{1},...,i_{m}=1}^{n}$ be a sequence of real scalars. If $%
0<r<2$, then there is a constant $C_{m,p}>0$ such that
\begin{eqnarray*}
&&\left( \sum_{i_{1},\dots ,i_{m}=1}^{n}|a_{i_{1},\dots ,i_{m}}|^{r}\right)
^{\frac{1}{r}}\, \\
&\leq &C_{m,p}n^{m\left( \frac{1}{r}-\frac{1}{2}\right) }\left(
\int_{[0,1]^{m}}\left\vert \sum_{i_{1},\dots ,i_{m}=1}^{n}a_{i_{1},\dots
,i_{m}}r_{i_{1}}(t_{1})\cdots r_{i_{m}}(t_{m})\right\vert ^{p}\,dt_{1}\cdots
dt_{m}\right) ^{\frac{1}{p}}
\end{eqnarray*}%
and the exponent $m\left( \frac{1}{r}-\frac{1}{2}\right) $ is optimal.
\end{theorem}

The main technicality in the proof of the above result arises in the search
of the optimality of the parameters. For this task we shall use, among other
results, a powerful and deep combinatorial probabilistic tool, called
Kahane--Salem--Zygmund inequality.

\section{Preliminaries}

We start off by recalling some terminology. By $c_{0}$ we denote the Banach
space of all real-valued sequences $\left( a_{j}\right) _{j=1}^{\infty }$
such that $\lim_{j\rightarrow \infty }a_{j}=0,$ endowed with the $\sup $
norm. For a multilinear form $T:c_{0}\times \cdots \times c_{0}\rightarrow
\mathbb{R}$ we denote, as usual,%
\begin{equation*}
\Vert T\Vert :=\sup \left\{ \left\vert T\left( x^{(1)},...,x^{(m)}\right)
\right\vert :\left\Vert x^{(j)}\right\Vert =1\text{ for all }%
j=1,...,m\right\} .
\end{equation*}%
For more details on the theory of multilinear forms on Banach spaces we
refer to \cite{mujica}. For the reader's convenience we also recall that the
topological dual of $c_{0}$, denoted by $\left( c_{0}\right) ^{\ast }$ is
isometrically isomorphic to the sequence space of absolutely summable
sequences $\ell _{1}.$

We shall recall three important tools of Probability Theory and multilinear
operators that will be crucial to prove Theorem \ref{7657} and Proposition %
\ref{7658}. The first one is the beautiful Kahane--Salem--Zygmund inequality
(see, for instance, \cite{alb} and \cite{capp} and the references therein):

\begin{theorem}[Kahane--Salem--Zygmund inequality]
\label{sss} Let $m,n\geq 1.$ There is a universal constant $K_{m}>0$,
depending only on $m$, and an $m$-linear form $T_{m,n}\colon c_{0}\times
\cdots \times c_{0}\rightarrow \mathbb{R}$ of the form
\begin{equation*}
T_{m,n}(z^{(1)},...,z^{(m)})=\displaystyle\sum_{i_{1},...,i_{m}=1}^{n}\pm
z_{i_{1}}^{(1)}\cdots z_{i_{m}}^{(m)}
\end{equation*}%
such that
\begin{equation*}
\Vert T_{m,n}\Vert \leq K_{m}n^{\frac{m+1}{2}}.
\end{equation*}
\end{theorem}

As it will be seen in the next section, we shall prove the optimality of
Theorem \ref{7657} by considering, for all $i_{m+1},$%
\begin{equation*}
a_{i_{1}...i_{m}}^{(i_{m+1})}=T_{m+1,n}\left(
e_{i_{1}},...,e_{i_{m+1}}\right) .
\end{equation*}%
To the proof the optimality of Proposition \ref{7658} we shall need a
different approach. We shall consider $m$-linear forms $R_{m}:c_{0}\times
\cdots \times c_{0}\rightarrow \mathbb{R}$ defined inductively by%
\begin{eqnarray*}
R_{2}(x^{(1)},x^{(2)})
&=&x_{1}^{(1)}x_{1}^{(2)}+x_{1}^{(1)}x_{2}^{(2)}+x_{2}^{(1)}x_{1}^{(2)}-x_{2}^{(1)}x_{2}^{(2)},
\\
R_{3}(x^{(1)},x^{(2)},x^{(3)}) &=&\left( x_{1}^{(1)}+x_{2}^{(1)}\right)
\left(
x_{1}^{(2)}x_{1}^{(3)}+x_{1}^{(2)}x_{2}^{(3)}+x_{2}^{(2)}x_{1}^{(3)}-x_{2}^{(2)}x_{2}^{(3)}\right)
\\
&&+\left( x_{1}^{(1)}-x_{2}^{(1)}\right) \left(
x_{3}^{(2)}x_{3}^{(3)}+x_{3}^{(2)}x_{4}^{(3)}+x_{4}^{(2)}x_{3}^{(3)}-x_{4}^{(2)}x_{4}^{(3)}\right) ,
\end{eqnarray*}%
and so on (for details we refer to \cite{nnn}), and consider, for all $%
i_{m+1},$%
\begin{equation*}
a_{i_{1}...i_{m}}^{(i_{m+1})}=R_{m+1}\left( e_{i_{1}},...,e_{i_{m+1}}\right)
.
\end{equation*}%
It shall be important to note (see \cite{nnn}) that that each $R_{m}$ is
composed by precisely $2^{2m-2}$ monomials and that
\begin{equation*}
\left\Vert R_{m}\right\Vert =2^{m-1}.
\end{equation*}%
It is also important for our purposes to note that each $R_{m}$ has exactly $%
2^{m-1}$ monomials involving the coordinates of the last variable $x^{(m)}.$

Finally, we need a \textquotedblleft multiple index\textquotedblright\
version of the Contraction Principle. We present a proof for the sake of
completeness.

\begin{lemma}
\label{778}For all positive integers $m,n$ and vectors $y_{i_{1},\dots
,i_{m}}$ in a Banach space $Y$, $i_{1},\dots ,i_{m}=1,\dots ,n$, we have
\begin{equation*}
\max_{\substack{ i_{k}=1,\dots ,n  \\ k=1,\dots ,m}}\left\Vert
y_{i_{1},\dots ,i_{m}}\right\Vert \leq \int_{\lbrack 0,1]^{m}}\left\Vert
\sum_{i_{1},\dots ,i_{m}=1}^{n}r_{i_{1}}(t_{1})\cdots
r_{i_{m}}(t_{m})y_{i_{1},\dots ,i_{m}}\right\Vert \,dt_{1}\cdots dt_{m}.
\end{equation*}
\end{lemma}

\begin{proof}
The case $m=1$ is the Contraction Principle (see \cite[Theorem 12.2]{diestel}%
). Let us suppose, as the induction step, that the result is valid for $m-1$%
. Thus, for all positive integers $i_{1},\dots ,i_{m}$, we have
\begin{align*}
& \int_{[0,1]^{m}}\left\Vert \sum_{i_{1},\dots
,i_{m}=1}^{n}r_{i_{1}}(t_{1})\cdots r_{i_{m}}(t_{m})y_{i_{1},\dots
,i_{m}}\right\Vert \,dt_{1}\cdots dt_{m} \\
& =\int_{[0,1]^{m-1}}\left( \int_{0}^{1}\left\Vert
\sum_{i_{1}=1}^{n}r_{i_{1}}(t_{1})\left( \sum_{i_{2},\dots
,i_{m}=1}^{n}r_{i_{2}}(t_{2})\cdots r_{i_{m}}(t_{m})y_{i_{1},\dots
,i_{m}}\right) \right\Vert \,dt_{1}\right) \,dt_{2}\cdots dt_{m} \\
& \geq \int_{\lbrack 0,1]^{m-1}}\left\Vert \sum_{i_{2},\dots
,i_{m}=1}^{n}r_{i_{2}}(t_{2})\cdots r_{i_{m}}(t_{m})y_{i_{1},\dots
,i_{m}}\right\Vert \,dt_{2}\cdots dt_{m} \\
& \geq \left\Vert y_{i_{1},\dots ,i_{m}}\right\Vert .
\end{align*}
\end{proof}

\section{The proof of the main theorem}

Let us first show that there is a $t_{m,p}>0$ and a certain constant $%
C_{m,p}>0$ such that%
\begin{eqnarray}
&&\left( \sum_{i_{1},\dots ,i_{m}=1}^{n}|a_{i_{1},\dots ,i_{m}}|^{r}\right)
^{\frac{1}{r}}  \label{hgg} \\
&\leq &C_{m,p}n^{t_{m,p}}\left( \int_{[0,1]^{m}}\left\vert \sum_{i_{1},\dots
,i_{m}=1}^{n}r_{i_{1}}(t_{1})\cdots r_{i_{m}}(t_{m})a_{i_{1},\dots
,i_{m}}\right\vert ^{p}\,dt_{1}\cdots dt_{m}\right) ^{1/p}  \notag
\end{eqnarray}%
for all sequences $\left( a_{i_{1},...,i_{m}}\right)
_{i_{1},...,i_{m}=1}^{n} $ and all $n.$

Let $s>0$ be such that $\frac{1}{r}=\frac{1}{2}+\frac{1}{s}$. By the H\"{o}%
lder inequality and (\ref{multik}) with $p=1$ we have%
\begin{eqnarray*}
\left( \sum_{i_{1},\dots ,i_{m}=1}^{n}|a_{i_{1},\dots ,i_{m}}|^{r}\right) ^{%
\frac{1}{r}}\  &\leq &\left( \sum_{i_{1},\dots ,i_{m}=1}^{n}|a_{i_{1},\dots
,i_{m}}|^{2}\right) ^{\frac{1}{2}}\cdot \left( \sum\limits_{{i_{1},\dots
,i_{m}=1}}^{n}1^{s}\right) ^{\frac{1}{s}} \\
&\leq &2^{\frac{m}{2}}\cdot \left( \int_{\lbrack 0,1]^{m}}\left\vert
\sum_{i_{1},\dots ,i_{m}=1}^{n}r_{i_{1}}(t_{1})\cdots
r_{i_{m}}(t_{m})a_{i_{1},\dots ,i_{m}}\right\vert \,dt_{1}\cdots
dt_{m}\right) \cdot n^{\frac{m}{s}} \\
&=&2^{\frac{m}{2}}n^{m\left( \frac{1}{r}-\frac{1}{2}\right) }\cdot
\int_{\lbrack 0,1]^{m}}\left\vert \sum_{i_{1},\dots
,i_{m}=1}^{n}r_{i_{1}}(t_{1})\cdots r_{i_{m}}(t_{m})a_{i_{1},\dots
,i_{m}}\right\vert \,dt_{1}\cdots dt_{m}.
\end{eqnarray*}%
Now we show that the best estimate for $t_{m,1}$ in (\ref{hgg}) is precisely
$m\left( \frac{1}{r}-\frac{1}{2}\right) .$ In fact, let $T_{m+1,n}$ be given
by the Kahane--Salem--Zygmund inequality (Theorem \ref{sss}). Since $\left(
c_{0}\right) ^{\ast }=\ell _{1}$ we have
\begin{eqnarray*}
&&\sum\limits_{i_{m+1}=1}^{n}\left(
\sum\limits_{i_{1},...,i_{m}=1}^{n}\left\vert T_{m+1,n}\left(
e_{i_{1}},...,e_{i_{m+1}}\right) \right\vert ^{r}\right) ^{\frac{1}{r}} \\
&\leq
&\sum\limits_{i_{m+1}=1}^{n}C_{m,p}n^{t_{m,1}}\int_{[0,1]^{m}}\left\vert
\sum_{i_{1},\dots ,i_{m}=1}^{n}r_{i_{1}}(t_{1})\cdots
r_{i_{m}}(t_{m})T_{m+1,n}\left( e_{i_{1}},e_{i_{2}},...,e_{i_{m+1}}\right)
\right\vert \,dt_{1}\cdots dt_{m} \\
&=&C_{m,p}n^{t_{m,1}}\int_{[0,1]^{m}}\sum\limits_{i_{m+1}=1}^{n}\left\vert
T_{m+1,n}\left(
\sum_{i_{1}=1}^{n}r_{i_{1}}(t_{1})e_{i_{1}},...,%
\sum_{i_{m}=1}^{n}r_{i_{m}}(t_{m})e_{i_{m}},e_{i_{m+1}}\right) \right\vert
\,dt_{1}\cdots dt_{m} \\
&\leq &C_{m,p}n^{t_{m,1}}\sup_{t_{1},...,t_{m}\in \lbrack 0,1]}\left\Vert
T_{m+1,n}\left(
\sum_{i_{1}=1}^{n}r_{i_{1}}(t_{1})e_{i_{1}},...,%
\sum_{i_{m}=1}^{n}r_{i_{m}}(t_{m})e_{i_{m}},\cdot \right) \right\Vert  \\
&\leq &C_{m}n^{t_{m,1}}K_{m+1}n^{\frac{m+2}{2}}.
\end{eqnarray*}%
On the other hand,
\begin{equation*}
\sum\limits_{i_{m+1}=1}^{n}\left(
\sum\limits_{i_{1},...,i_{m}=1}^{n}\left\vert T_{m+1,n}\left(
e_{i_{1}},...,e_{i_{m+1}}\right) \right\vert ^{r}\right) ^{\frac{1}{r}%
}=n\cdot n^{\frac{m}{r}}.
\end{equation*}%
Hence%
\begin{equation*}
n^{1+\frac{m}{r}}\leq C_{m,p}n^{t_{m,1}}K_{m+1}n^{\frac{m+2}{2}}
\end{equation*}%
for all $n$. Since $n$ is arbitrary, we have%
\begin{equation*}
t_{m,1}\geq m\left( \frac{1}{r}-\frac{1}{2}\right) .
\end{equation*}%
By \cite{popa} we know that for any $p,q>0$ and all positive integers $m,$
there is a constant $C_{m,p,q}>0$ such that
\begin{eqnarray*}
&&\left( \int_{[0,1]^{m}}\left\vert \sum_{i_{1},\dots
,i_{m}=1}^{n}a_{i_{1},\dots
,i_{m}}\prod\limits_{j=1}^{m}r_{i_{j}}(t_{j})\right\vert ^{p}\,dt_{1}\cdots
dt_{m}\right) ^{\frac{1}{p}} \\
&\leq &C_{m,p,q}\left( \int_{[0,1]^{m}}\left\vert \sum_{i_{1},\dots
,i_{m}=1}^{n}a_{i_{1},\dots
,i_{m}}\prod\limits_{j=1}^{m}r_{i_{j}}(t_{j})\right\vert ^{q}\,dt_{1}\cdots
dt_{m}\right) ^{\frac{1}{q}}
\end{eqnarray*}%
and thus we conclude that the optimal $t_{m,p}$ coincides with the optimal $%
t_{m,1},$ regardless of the $p>0$, and the proof is done.

\begin{remark}
\label{hh77}If $0<r_{j}<2$ for all $j=1,...,m$, using the mixed H\"{o}lder
inequality (see \cite{ff}) and repeating the arguments of the proof of
Theorem \ref{7657} we can prove that there is a constant $C_{m,p}>0$ such
that%
\begin{eqnarray*}
&&\left( \sum_{i_{1}=1}^{n}\left( \sum\limits_{i_{2}=1}^{n}\left( \cdots
\left( \sum\limits_{i_{m}=1}^{n}\left\vert a_{i_{1},\dots ,i_{m}}\right\vert
^{r_{m}}\right) ^{\frac{1}{r_{m}}}\cdots \right) ^{\frac{1}{r_{3}}\times
r_{2}}\right) ^{\frac{1}{r_{2}}\times r_{1}}\right) ^{\frac{1}{r_{1}}} \\
&\leq &C_{m,p}\cdot n^{\left( \sum_{j=1}^{m}\frac{1}{r_{j}}\right) -\frac{m}{%
2}}\left( \int_{[0,1]^{m}}\left\vert \sum_{i_{1},\dots
,i_{m}=1}^{n}r_{i_{1}}(t_{1})\cdots r_{i_{m}}(t_{m})a_{i_{1},\dots
,i_{m}}\right\vert ^{p}\,dt_{1}\cdots dt_{m}\right) ^{1/p}
\end{eqnarray*}%
and that the exponent $\left( \sum_{j=1}^{n}\frac{1}{r_{j}}\right) -\frac{m}{%
2}$ is sharp.
\end{remark}

\section{Optimal constants for variants of the Khinchin inequality}

We begin this section by providing the optimal constants satisfying (\ref%
{multik}) when $p=1$ and $r\geq 2:$

\begin{proposition}
\label{7658}Let $m,n$ be positive integers and $\left( a_{i_{1},\dots
,i_{m}}\right) _{i_{1},...,i_{m}=1}^{n}$ be a sequence of real scalars. If $%
r\geq 2$, then
\begin{equation}
\left( \sum_{i_{1},\dots ,i_{m}=1}^{n}|a_{i_{1},\dots ,i_{m}}|^{r}\right) ^{%
\frac{1}{r}}\,\leq 2^{\frac{m}{r}}\int_{[0,1]^{m}}\left\vert
\sum_{i_{1},\dots ,i_{m}=1}^{n}r_{i_{1}}(t_{1})\cdots
r_{i_{m}}(t_{m})a_{i_{1},\dots ,i_{m}}\right\vert \,dt_{1}\cdots dt_{m}
\label{st}
\end{equation}%
and the estimate $2^{\frac{m}{r}}$ is optimal.
\end{proposition}

Let us denote by $C_{r}$ the optimal constant satisfying%
\begin{equation}
\left( \sum_{i_{1},\dots ,i_{m}=1}^{n}|a_{i_{1},\dots ,i_{m}}|^{r}\right) ^{%
\frac{1}{r}}\,\leq C_{r}\int_{[0,1]^{m}}\left\vert \sum_{i_{1},\dots
,i_{m}=1}^{n}r_{i_{1}}(t_{1})\cdots r_{i_{m}}(t_{m})a_{i_{1},\dots
,i_{m}}\right\vert \,dt_{1}\cdots dt_{m}  \label{jj}
\end{equation}%
for all sequence of scalars $\left( a_{i_{1},\dots ,i_{m}}\right)
_{i_{1},\dots ,i_{m}=1}^{n},$ for all $n.$ Let $\theta =\frac{2}{r};$ by the
H\"{o}lder inequality, (\ref{multik}) and Lemma \ref{778} we conclude that%
\begin{align*}
\left( \sum_{i_{1},\dots ,i_{m}=1}^{n}|a_{i_{1},\dots ,i_{m}}|^{r}\right) ^{%
\frac{1}{r}}& \leq \left( \sum_{i_{1},\dots ,i_{m}=1}^{n}|a_{i_{1},\dots
,i_{m}}|^{2}\right) ^{\frac{\theta }{2}}\cdot \left( \max_{\substack{ %
i_{k}=1,\dots ,n  \\ k=1,\dots ,m}}\left\vert a_{i_{1},\dots
,i_{m}}\right\vert \right) ^{1-\theta } \\
& \leq 2^{\frac{m\theta }{2}}\int_{[0,1]^{m}}\left\vert \sum_{i_{1},\dots
,i_{m}=1}^{n}r_{i_{1}}(t_{1})\cdots r_{i_{m}}(t_{m})a_{i_{1},\dots
,i_{m}}\right\vert \,dt_{1}\cdots dt_{m} \\
& =2^{\frac{m}{r}}\int_{[0,1]^{m}}\left\vert \sum_{i_{1},\dots
,i_{m}=1}^{n}r_{i_{1}}(t_{1})\cdots r_{i_{m}}(t_{m})a_{i_{1},\dots
,i_{m}}\right\vert \,dt_{1}\cdots dt_{m},
\end{align*}%
Now let us prove that the constant $2^{\frac{m}{r}}$ is sharp. Let $R_{m+1}$
be the $m+1$-linear form defined in the Section 2. Using that $\left(
c_{0}\right) ^{\ast }=\ell _{1}$, we have%
\begin{align*}
& \sum\limits_{i_{m+1}=1}^{2^{m}}\left( \sum\limits_{i_{1},\dots
,i_{m}=1}^{2^{m}}\left\vert R_{m+1}\left(
e_{i_{1}},e_{i_{2}},...,e_{i_{m+1}}\right) \right\vert ^{r}\right) ^{\frac{1%
}{r}} \\
& \leq \sum\limits_{i_{m+1}=1}^{2^{m}}C_{r}\int_{[0,1]^{m}}\left\vert
\sum_{i_{1},\dots ,i_{m}=1}^{2^{m}}r_{i_{1}}(t_{1})\cdots
r_{i_{m}}(t_{m})R_{m+1}\left( e_{i_{1}},e_{i_{2}},...,e_{i_{m+1}}\right)
\right\vert \,dt_{1}\cdots dt_{m} \\
& =C_{r}\int_{[0,1]^{m}}\sum\limits_{i_{m+1}=1}^{2^{m}}\left\vert
R_{m+1}\left(
\sum_{i_{1}=1}^{2^{m}}r_{i_{1}}(t_{1})e_{i_{1}},...,%
\sum_{i_{m}=1}^{2^{m}}r_{i_{m}}(t_{m})e_{i_{m}},e_{i_{m+1}}\right)
\right\vert \,dt_{1}\cdots dt_{m} \\
& \leq C_{r}\sup_{t_{1},..,t_{m\in \lbrack
0,1]}}\sum\limits_{i_{m+1}=1}^{2^{m}}\left\vert R_{m+1}\left(
\sum_{i_{1}=1}^{2^{m}}r_{i_{1}}(t_{1})e_{i_{1}},...,%
\sum_{i_{m}=1}^{2^{m}}r_{i_{m}}(t_{m})e_{i_{m}},e_{i_{m+1}}\right)
\right\vert \\
& \leq 2^{m}C_{r}.
\end{align*}%
On the other hand, since $R_{m+1}$ has exactly $2^{m}$ monomials involving
the coordinates of the last variable and since $R_{m+1}$ has a total of $%
2^{2m}$ monomials, we conclude that%
\begin{equation*}
\sum\limits_{i_{m+1}=1}^{2^{m}}\left( \sum\limits_{i_{1},\dots
,i_{m}=1}^{2^{m}}\left\vert R_{m+1}\left(
e_{i_{1}},e_{i_{2}},...,e_{i_{m+1}}\right) \right\vert ^{r}\right) ^{\frac{1%
}{r}}=2^{m}\cdot (2^{m})^{\frac{1}{r}}.
\end{equation*}%
Thus,
\begin{equation*}
2^{m}\cdot 2^{\frac{m}{r}}\leq 2^{m}C_{r}
\end{equation*}%
and we obtain
\begin{equation*}
C_{r}\geq 2^{\frac{m}{r}},
\end{equation*}%
completing the proof.

\begin{remark}
It sounds reasonable that there exists a more direct proof of Proposition %
\ref{7658}. However, the fact that in general $\left(
\prod\limits_{j=1}^{m}r_{i_{j}}:[0,1]^{m}\rightarrow \mathbb{R}\right)
_{i_{1},...,i_{m}=1}^{\infty }$ is not independent may be an additional
difficulty.
\end{remark}

\section{Blow up rate of Kahane type inequalities}

Let $2\leq q<\infty $ and $s>0$. A Banach space $Y$ has cotype $q$ (see \cite{diestel, perez}) if there
is a constant $C>0$ such that, no matter how we select finitely many vectors
$y_{1},\dots ,y_{n}\in Y$,
\begin{equation}
\left( \sum_{k=1}^{n}\Vert y_{k}\Vert ^{q}\right) ^{\frac{1}{q}}\leq C\left(
\int_{[0,1]}\left\Vert \sum_{k=1}^{n}r_{k}(t)y_{k}\right\Vert ^{s}dt\right)
^{\frac{1}{s}}.
\end{equation}%
The smallest of all these constants is denoted by $C_{q}(Y)$ when $s=2$ and $%
c_{q}(Y)$ when $s=q$. The Kahane inequality (below) shows that the choice of
$s$ is not relevant (modulo the constant involved):

\begin{theorem}[Kahane Inequality]
If $0<p,q<\infty ,$ then there is a constant $K_{p,q}>0$ for which
\begin{equation*}
\left( \int_{\lbrack 0,1]}\left\Vert \sum_{k=1}^{n}r_{k}(t)y_{k}\right\Vert
^{q}dt\right) ^{\frac{1}{q}}\leq K_{p,q}\left( \int_{[0,1]}\left\Vert
\sum_{k=1}^{n}r_{k}(t)y_{k}\right\Vert ^{p}dt\right) ^{\frac{1}{p}}
\end{equation*}%
holds, regardless of the choice of a Banach space $Y$ and of finitely many
vectors $y_{1},\dots ,y_{n}\in Y$.
\end{theorem}

From now on $K_{p,q}$ denotes the optimal constant of the Kahane inequality.
As it happens for the Khinchin inequality, we have a Kahane inequality for
multiple indexes (see, for instance, \cite{achour}):

\begin{theorem}[Multiple Kahane Inequality]
\label{6f}If $0<p,q<\infty ,$ then
\begin{eqnarray*}
&&\left( \int_{\lbrack 0,1]^{m}}\left\Vert \sum_{i_{1},\dots
,i_{m}=1}^{n}y_{i_{1},\dots ,i_{m}}r_{i_{1}}(t_{1})\cdots
r_{i_{m}}(t_{m})\right\Vert ^{q}dt_{1}...dt_{m}\right) ^{\frac{1}{q}} \\
&\leq &K_{p,q}^{m}\left( \int_{[0,1]^{m}}\left\Vert \sum_{i_{1},\dots
,i_{m}=1}^{n}y_{i_{1},\dots ,i_{m}}r_{i_{1}}(t_{1})\cdots
r_{i_{m}}(t_{m})\right\Vert ^{p}dt_{1}...dt_{m}\right) ^{\frac{1}{p}},
\end{eqnarray*}%
for all Banach spaces $Y$ and all $y_{i_{1},\dots ,i_{m}}$ in $Y$.
\end{theorem}

The following result shows how cotype $q$ spaces behave with sums in
multiple indexes (see, for instance, \cite[Lemma 3.9]{perez}):

\begin{theorem}[Multiple cotype inequality]
\label{prop3dan} Let $Y$ be a cotype $q$ space. If $\left( y_{i_{1}\ldots
i_{m}}\right) _{i_{1},\cdots ,i_{m}=1}^{n}$ is a matrix in $Y$, then
\begin{equation*}
\left( \sum_{i_{1},\ldots ,i_{m}=1}^{n}\left\Vert y_{i_{1}\cdots
i_{m}}\right\Vert ^{q}\right) ^{1/q}\leq c_{q}(Y)^{m}\left(
\int_{[0,1]^{m}}\left\Vert \sum_{i_{1},...,i_{m}=1}^{n}y_{i_{1}\ldots
i_{m}}r_{i_{1}}(t_{1})\cdots r_{i_{m}}(t_{m})\right\Vert ^{q}dt_{1}\cdots
dt_{m}\right) ^{1/q}.
\end{equation*}
\end{theorem}

By the multiple Kahane inequality it is plain that from the above inequality
we have%
\begin{eqnarray}
&&\left( \sum_{i_{1},\ldots ,i_{m}=1}^{n}\left\Vert y_{i_{1}\cdots
i_{m}}\right\Vert ^{q}\right) ^{1/q}  \label{8u} \\
&\leq &c_{q}(Y)^{m}K_{s,q}^{m}\left( \int_{[0,1]^{m}}\left\Vert
\sum_{i_{1},...,i_{m}=1}^{n}y_{i_{1}\ldots i_{m}}r_{i_{1}}(t_{1})\cdots
r_{i_{m}}(t_{m})\right\Vert ^{s}dt_{1}\cdots dt_{m}\right) ^{\frac{1}{s}}
\notag
\end{eqnarray}%
for all $s>0.$ Our next result shows how is the exact blow up rate of the
constant arising when we consider cotype $2$ spaces replacing the $\ell _{2}$
norm by a $\ell _{r}$ norm, $r<2,$ in the left hand side of the above
inequality.

\begin{theorem}
Let $Y\neq \{0\}$ be a cotype $2$ space and $p>0$. If $0<r\leq 2$ and $%
\left( y_{i_{1}\ldots i_{m}}\right) _{i_{1},\cdots ,i_{m}=1}^{n}$ is a
matrix in $Y$, then there is a constant $c_{m,p}>0$ such that
\begin{eqnarray*}
&&\left( \sum_{i_{1},\ldots ,i_{m}=1}^{n}\left\Vert y_{i_{1}\cdots
i_{m}}\right\Vert ^{r}\right) ^{1/r} \\
&\leq &c_{m,p}n^{m\left( \frac{1}{r}-\frac{1}{2}\right) }\left(
\int_{[0,1]^{m}}\left\Vert \sum_{i_{1},...,i_{m}=1}^{n}y_{i_{1}\ldots
i_{m}}r_{i_{1}}(t_{1})\cdots r_{i_{m}}(t_{m})\right\Vert ^{p}dt_{1}\cdots
dt_{m}\right) ^{\frac{1}{p}}
\end{eqnarray*}%
and the exponent $m\left( \frac{1}{r}-\frac{1}{2}\right) $ is optimal.
\end{theorem}

\begin{proof}
As in the proof of Theorem \ref{7657}, we have
\begin{eqnarray}
&&\left( \sum_{i_{1},\dots ,i_{m}=1}^{n}\left\Vert y_{i_{1},\dots
,i_{m}}\right\Vert ^{r}\right) ^{\frac{1}{r}}\   \label{8um} \\
&\leq &c_{2}(Y)^{m}K_{1,2}^{m}n^{m\left( \frac{1}{r}-\frac{1}{2}\right)
}\cdot \int_{\lbrack 0,1]^{m}}\left\Vert \sum_{i_{1},\dots
,i_{m}=1}^{n}r_{i_{1}}(t_{1})\cdots r_{i_{m}}(t_{m})y_{i_{1},\dots
,i_{m}}\right\Vert \,dt_{1}\cdots dt_{m}.  \notag
\end{eqnarray}

To prove the optimality of the above exponent $m\left( \frac{1}{r}-\frac{1}{2%
}\right) $, let us suppose that%
\begin{equation*}
\left( \sum_{i_{1},\dots ,i_{m}=1}^{n}\left\Vert y_{i_{1},\dots
,i_{m}}\right\Vert ^{r}\right) ^{\frac{1}{r}}\ \leq c_{m}n^{t}\cdot
\int_{\lbrack 0,1]^{m}}\left\Vert \sum_{i_{1},\dots
,i_{m}=1}^{n}r_{i_{1}}(t_{1})\cdots r_{i_{m}}(t_{m})y_{i_{1},\dots
,i_{m}}\right\Vert \,dt_{1}\cdots dt_{m}
\end{equation*}%
for a certain $c_{m}>0.$ Consider the $m+1$-linear form $T_{m+1,n}$ given by
the Kahane--Salem--Zygmund inequality and define%
\begin{equation*}
S_{m+1,n}(x_{1},...,x_{m+1})=T_{m+1,n}(x_{1},...,x_{m+1})y,
\end{equation*}%
for a certain fixed $y\in Y$ with $\left\Vert y\right\Vert =1.$ Then
\begin{eqnarray*}
&&\sum\limits_{i_{m+1}=1}^{n}\left(
\sum\limits_{i_{1},...,i_{m}=1}^{n}\left\Vert S_{m+1,n}\left(
e_{i_{1}},...,e_{i_{m+1}}\right) \right\Vert ^{r}\right) ^{\frac{1}{r}} \\
&\leq &\sum\limits_{i_{m+1}=1}^{n}c_{m}n^{t}\int_{[0,1]^{m}}\left\vert
\sum_{i_{1},\dots ,i_{m}=1}^{n}r_{i_{1}}(t_{1})\cdots
r_{i_{m}}(t_{m})T_{m+1,n}\left( e_{i_{1}},e_{i_{2}},...,e_{i_{m+1}}\right)
\right\vert \,dt_{1}\cdots dt_{m} \\
&\leq &c_{m}n^{t}K_{m+1}n^{\frac{m+2}{2}}.
\end{eqnarray*}%
Proceeding again as in the proof of Theorem \ref{7657} we conclude that%
\begin{equation*}
t\geq m\left( \frac{1}{r}-\frac{1}{2}\right) .
\end{equation*}%
By Theorem \ref{6f} we know that the same optimal estimate holds when
replacing the $L_{1}$-norm in (\ref{8um}) by any $L_{p}$-norm.
\end{proof}

\begin{remark}
A result similar to the one stated in Remark \ref{hh77} applies for this
case of cotype $2$ spaces.
\end{remark}

When $Y$ is a Hilbert space we can prove a result similar to Proposition \ref%
{7658}:

\begin{theorem}
Let $m,n$ be positive integers and $\left( y_{i_{1},\dots ,i_{m}}\right)
_{i_{1},...,i_{m}=1}^{n}$ be a sequence in a Hilbert space $Y.$ If $r\geq 2$%
, then%
\begin{equation*}
\left( \sum_{i_{1},\dots ,i_{m}=1}^{n}\left\Vert y_{i_{1},\dots
,i_{m}}\right\Vert ^{r}\right) ^{\frac{1}{r}}\,\leq 2^{\frac{m}{r}%
}\int_{[0,1]^{m}}\left\Vert \sum_{i_{1},\dots
,i_{m}=1}^{n}r_{i_{1}}(t_{1})\cdots r_{i_{m}}(t_{m})y_{i_{1},\dots
,i_{m}}\right\Vert \,dt_{1}\cdots dt_{m}
\end{equation*}%
and the constant $2^{\frac{m}{r}}$ is optimal.
\end{theorem}

\begin{proof}
The proof is similar to the proof of Proposition \ref{7658}$.$ For $r\geq 2$
let us denote by $C_{r}$ the optimal constant satisfying%
\begin{equation*}
\left( \sum_{i_{1},\dots ,i_{m}=1}^{n}\left\Vert y_{i_{1},\dots
,i_{m}}\right\Vert ^{r}\right) ^{\frac{1}{r}}\,\leq
C_{r}\int_{[0,1]^{m}}\left\Vert \sum_{i_{1},\dots
,i_{m}=1}^{n}r_{i_{1}}(t_{1})\cdots r_{i_{m}}(t_{m})y_{i_{1},\dots
,i_{m}}\right\Vert \,dt_{1}\cdots dt_{m}
\end{equation*}%
for all sequence of scalars $\left( y_{i_{1},\dots ,i_{m}}\right)
_{i_{1},\dots ,i_{m}=1}^{n},$ for all $n.$ Let $\theta =\frac{2}{r};$ since $%
K_{1,2}=\sqrt{2}$ and $c_{2}(Y)=1$ (see \cite{latala} and \cite[Corollary
11.8]{diestel}), by the H\"{o}lder inequality, (\ref{8u}) and Lemma \ref{778}
we conclude that%
\begin{align*}
\left( \sum_{i_{1},\dots ,i_{m}=1}^{n}\left\Vert y_{i_{1},\dots
,i_{m}}\right\Vert ^{r}\right) ^{\frac{1}{r}}\,& \leq \left(
\sum_{i_{1},\dots ,i_{m}=1}^{n}\left\Vert y_{i_{1},\dots ,i_{m}}\right\Vert
^{2}\right) ^{\frac{\theta }{2}}\,\cdot \left( \max_{\substack{ %
i_{k}=1,\dots ,n  \\ k=1,\dots ,m}}\left\Vert y_{i_{1},\dots
,i_{m}}\right\Vert \right) ^{1-\theta } \\
& \leq \left( c_{2}(Y)^{m}K_{1,2}^{m}\right) ^{\theta
}\int_{[0,1]^{m}}\left\Vert \sum_{i_{1},\dots
,i_{m}=1}^{n}r_{i_{1}}(t_{1})\cdots r_{i_{m}}(t_{m})y_{i_{1},\dots
,i_{m}}\right\Vert \,dt_{1}\cdots dt_{m} \\
& =2^{\frac{m}{r}}\int_{[0,1]^{m}}\left\Vert \sum_{i_{1},\dots
,i_{m}=1}^{n}r_{i_{1}}(t_{1})\cdots r_{i_{m}}(t_{m})y_{i_{1},\dots
,i_{m}}\right\Vert \,dt_{1}\cdots dt_{m},
\end{align*}%
Now let us prove that the constant $2^{\frac{m}{r}}$ is sharp. Let $S_{m+1}$
be the $m+1$-linear form $R_{m+1}$ defined in Section 2, multiplied by a
fixed unit vector $y\in Y$. Using that $\left( c_{0}\right) ^{\ast }=\ell
_{1}$, we have%
\begin{align*}
& \sum\limits_{i_{m+1}=1}^{2^{m}}\left( \sum\limits_{i_{1},\dots
,i_{m}=1}^{2^{m}}\left\Vert S_{m+1}\left(
e_{i_{1}},e_{i_{2}},...,e_{i_{m+1}}\right) \right\Vert ^{r}\right) ^{\frac{1%
}{r}} \\
& \leq \sum\limits_{i_{m+1}=1}^{2^{m}}C_{r}\int_{[0,1]^{m}}\left\vert
\sum_{i_{1},\dots ,i_{m}=1}^{2^{m}}r_{i_{1}}(t_{1})\cdots
r_{i_{m}}(t_{m})R_{m+1}\left( e_{i_{1}},e_{i_{2}},...,e_{i_{m+1}}\right)
\right\vert \,dt_{1}\cdots dt_{m} \\
& =C_{r}\int_{[0,1]^{m}}\sum\limits_{i_{m+1}=1}^{2^{m}}\left\vert
R_{m+1}\left(
\sum_{i_{1}=1}^{2^{m}}r_{i_{1}}(t_{1})e_{i_{1}},...,%
\sum_{i_{m}=1}^{2^{m}}r_{i_{m}}(t_{m})e_{i_{m}},e_{i_{m+1}}\right)
\right\vert \,dt_{1}\cdots dt_{m} \\
& \leq C_{r}\sup_{t_{1},..,t_{m\in \lbrack
0,1]}}\sum\limits_{i_{m+1}=1}^{2^{m}}\left\vert R_{m+1}\left(
\sum_{i_{1}=1}^{2^{m}}r_{i_{1}}(t_{1})e_{i_{1}},...,%
\sum_{i_{m}=1}^{2^{m}}r_{i_{m}}(t_{m})e_{i_{m}},e_{i_{m+1}}\right)
\right\vert \\
& \leq 2^{m}C_{r}.
\end{align*}%
On the other hand, since $R_{m+1}$ has exactly $2^{m}$ monomials involving
the coordinates of the last variable and since $R_{m+1}$ has a total of $%
2^{2m}$ monomials, we conclude that%
\begin{equation*}
\sum\limits_{i_{m+1}=1}^{2^{m}}\left( \sum\limits_{i_{1},\dots
,i_{m}=1}^{2^{m}}\left\Vert R_{m+1}\left(
e_{i_{1}},e_{i_{2}},...,e_{i_{m+1}}\right) y\right\Vert ^{r}\right) ^{\frac{1%
}{r}}=2^{m}\cdot (2^{m})^{\frac{1}{r}}.
\end{equation*}%
and the proof is concluded as in Proposition \ref{7658}.
\end{proof}

\end{document}